\documentclass[twoside]{article}
\usepackage{amsmath,amssymb,amsthm}
\numberwithin{equation}{section}

\theoremstyle{plain}
\newtheorem{thm}{Theorem}[section]
\newtheorem{lem}[thm]{Lemma}
\newtheorem{cor}[thm]{Corollary}

\theoremstyle{remark}
\newtheorem{remark}[thm]{Remark}

\newenvironment{ackn}{\subsection*{Acknowledgement}}{}
\newenvironment{keywords}{\smallskip\noindent{\bfseries Keywords:}}{}
\newenvironment{MSC}{\smallskip\noindent{\bfseries 2000 Mathematical Subject Classification:}}{}

\DeclareMathOperator{\per}{per} 
\DeclareMathOperator{\rank}{rk} 
\DeclareMathOperator{\Sym}{Sym} 

\newcommand{\Z}{\mathbb{Z}} 
\newcommand{\C}{\mathbb{C}} 
\newcommand{\gl}{\mathfrak{gl}} 

\newcommand{\Ugl}[1][n]{\mathcal{U}(\mathfrak{gl}_{#1})} 
\newcommand{\Uq}[1][2]{\mathcal{U}_q(\mathfrak{sl}_{#1})} 
\newcommand{\A}[1][2]{\mathcal{A}(\operatorname{Mat}_{#1})} 
\newcommand{\Aq}[1][2]{\mathcal{A}_q(\operatorname{Mat}_{#1})} 
\newcommand{\Uqn}[1][n]{\mathcal{U}_q(\mathfrak{gl}_{#1})} 

\newcommand{\GLmod}[2][2]{\mathcal{M}_{#1}^{#2}} 
\newcommand{\sym}[1]{\mathfrak{S}_{#1}} 

\newcommand{\SLmod}[1]{\mathcal{M}_q(#1)} 
\newcommand{\Cmod}[2][\alpha]{V_q^{#2}(#1)} 

\newcommand{\qdet}[1][q]{\operatorname{det}_{#1}} 
\newcommand{\qper}[1][q]{\operatorname{per}_{#1}} 
\newcommand{\adet}[1][\alpha]{\operatorname{det}^{(#1)}} 
\newcommand{\qadet}[1][\alpha]{\operatorname{det}_q^{(#1)}} 
\newcommand{\len}[1]{\ell(#1)} 
\newcommand{\inv}[1]{\ell(#1)} 

\newcommand{\mixhgf}[6][q]{\Phi\!\left(\genfrac{}{}{0pt}{}{#2}{#3}\,;\,\genfrac{}{}{0pt}{}{#4}{#5}\,;\,#1\,;\,#6\right)} 
\newcommand{\genhgf}[5]{\mathchoice{{}_{#1}F_{#2}\!\left(\genfrac{}{}{0pt}{}{#3}{#4}\,;\,#5\right)}{{}_{#1}F_{#2}(#3;#4;#5)}{}{}} 

\newcommand{\qbinom}[3][q]{\genfrac{[}{]}{0pt}{}{#2}{#3}_{#1}} 
\newcommand{\q}[2][q]{\left[#2\right]_{#1}} 
\newcommand{\ps}[2]{\left(#1;#2\right)} 
\newcommand{\qps}[3][q]{\left(#2;#3\right)_{#1}} 

\newcommand{\kakko}[1]{\left(#1\right)} 
\newcommand{\ckakko}[1]{\left\{#1\right\}} 
\newcommand{\eqsp}{\phantom{{}={}}}

\begin{document}

\title{\bfseries Quantum alpha-determinants\\ and $q$-deformed hypergeometric polynomials}
\author{Kazufumi Kimoto}
\date{February 26, 2009}

\maketitle

\begin{abstract}
The quantum $\alpha$-determinant is defined as a parametric deformation of the quantum determinant.
We investigate the cyclic $\mathcal{U}_q(\mathfrak{sl}_2)$-submodules
of the quantum matrix algebra $\mathcal{A}_q(\mathrm{Mat}_2)$ generated by
the powers of the quantum $\alpha$-determinant.
For such a cyclic module, there exists a collection of polynomials
which describe the irreducible decomposition of it in the following manner:
(i) each polynomial corresponds to a certain irreducible $\mathcal{U}_q(\mathfrak{sl}_2)$-module,
(ii) the cyclic module contains an irreducible submodule if the parameter is a root of the corresponding polynomial.
These polynomials are given as a $q$-deformation of the hypergeometric polynomials.
This is a quantum analogue of the result obtained in our previous work
[K. Kimoto, S. Matsumoto and M. Wakayama,
\textit{Alpha-determinant cyclic modules and Jacobi polynomials},
to appear in Trans. Amer. Math. Soc.].


\begin{MSC}
20G42, 
33C20. 
\end{MSC}
\end{abstract}

\section{Introduction}

The \emph{$\alpha$-determinant} is a common generalization of the determinant and permanent defined by
\begin{equation*}
\adet(X)=\sum_{\sigma\in\sym n}\alpha^{\nu(\sigma)}x_{\sigma(1)1}x_{\sigma(2)2}\dotsb x_{\sigma(n)n},
\end{equation*}
where $\alpha$ is a complex parameter and $\nu(\sigma)=n-(m_1+m_2+\dotsb+m_n)$
if the cycle type of $\sigma\in\sym n$ is $1^{m_1}2^{m_2}\dotsb n^{m_n}$ \cite{VereJones}.
By definition, the $\alpha$-determinant $\adet(X)$ agrees with the determinant $\det(X)$ when $\alpha=-1$
and with the permanent $\per(X)$ when $\alpha=1$.
In other words, the $\alpha$-determinant interpolates these two.

We recall an invariant property of the determinant and permanent in the following representation-theoretical context.
Let $\A[n]$ be the $\C$-algebra of polynomials in the $n^2$ commuting variables $\{x_{ij}\}_{1\le i,j\le n}$,
and $\Ugl$ the universal enveloping algebra of the Lie algebra $\gl_n=\gl_n(\C)$.
By defining
\begin{align*}
E_{ij}\cdot f=\sum_{r=1}^n x_{ir}\frac{\partial f}{\partial x_{jr}}\quad(f\in\A[n]),
\end{align*}
$\A[n]$ becomes a $\Ugl$-module.
Here $\{E_{ij}\}_{1\le i,j\le n}$ is the standard basis of $\gl_n$.
Then, both of the determinant $\det(X)$ and the permanent $\per(X)$ generate irreducible $\Ugl$-submodules of $\A[n]$.
In fact, the cyclic submodules $\Ugl\cdot\det(X)$ and $\Ugl\cdot\per(X)$ are equivalent to
the skew-symmetric tensor product $\wedge^n(\C^n)$ and symmetric tensor product $\Sym^n(\C^n)$
of the natural representation $\C^n$ respectively, which are irreducible.

In view of this fact,
it is natural and interesting to study the irreducible decomposition
of the cyclic submodule $\Ugl\cdot\adet(X)$, or more generally $\Ugl\cdot\adet(X)^m$.
Matsumoto and Wakayama \cite{MW} tackled this problem first
and obtained explicit irreducible decomposition of $\Ugl\cdot\adet(X)$,
and recently Matsumoto, Wakayama and the author investigated the general case $\Ugl\cdot\adet(X)^m$ \cite{KMW:TAMS};
It is proved that
\begin{align*}
\Ugl\cdot\adet(X)^m\cong\bigoplus_{\substack{\lambda\vdash mn\\ \len\lambda\le n}}\kakko{\GLmod[n]\lambda}^{\oplus \rank F_{n,m}^\lambda(\alpha)}
\end{align*}
holds for certain square matrices $F_{n,m}^\lambda(\alpha)$ whose entries are polynomials in $\alpha$.
In this direct sum, $\lambda$ runs over the partitions of $mn$ whose length is at most $n$.
Here we identify the dominant integral weights and partitions,
and denote by $\GLmod[n]\lambda$ the irreducible representation of $\Ugl$ with highest weight $\lambda$.
Remark that the matrices $F_{n,m}^\lambda(\alpha)$ are determined up to conjugacy and non-zero scalar factor.
In the particular case where $m=1$, we explicitly have
$F_{n,1}^\lambda(\alpha)=f_\lambda(\alpha)I$,
where $I$ is the identity matrix and $f_\lambda(\alpha)$ is the (modified) content polynomial for $\lambda$ \cite{MW}.
It seems quite difficult to obtain an explicit expression for $F_{n,m}^\lambda(\alpha)$ in general.
However, when $n=2$, all the matrices $F_{2,m}^\lambda(\alpha)$ are \emph{one by one}, and they are explicitly given by
\begin{equation}\label{eq:classical_transition_polynomials}
F_{2,m}^{(m+s,m-s)}(\alpha)=(1+\alpha)^s\genhgf21{s-m,s+1}{-m}{-\alpha}\qquad(s=0,1,\dots,m),
\end{equation}
where $\genhgf21{a,b}cx$ is the \emph{Gaussian hypergeometric function} \cite{KMW:TAMS}.

These problems can be also formulated in the framework of \emph{quantum groups}.
Namely, we define a quantum counterpart of the $\alpha$-determinant, which we call \emph{quantum $\alpha$-determinant}, by
\begin{equation}\label{eq:def_of_qadet}
\qadet=\sum_{\sigma\in\sym n}\alpha^{\nu(\sigma)}q^{\inv\sigma}x_{\sigma(1)1}x_{\sigma(2)2}\dotsb x_{\sigma(n)n}
\end{equation}
in the \emph{quantum matrix algebra} $\Aq[n]$ \cite{RTF1990}.
Here $\inv\sigma$ denotes the inversion number of a permutation $\sigma$.
This agrees with the quantum determinant when $\alpha=-1$.
We then introduce a $\Uqn$-module structure on it,
where $\Uqn$ is the quantum enveloping algebra of $\gl_n$ \cite{Jimbo1986},
and consider the cyclic module $\Uqn\cdot(\qadet)^m$.
In \cite{KWquantum}, we study the case where $m=1$.
In contrast to the classical case, the structure of the cyclic module is much more complicated,
so that we have only obtained several less explicit results.

Nevertheless, we can establish a quantum version of the result \eqref{eq:classical_transition_polynomials} \emph{rather explicitly},
and this is the aim of the present article.
We investigate the cyclic $\Uq$-submodule (instead of $\Uqn[2]$-submodule just for simplicity of the description) of $\Aq$ defined by
\begin{align*}
\Cmod m=\Uq\cdot(\qadet)^m.
\end{align*}
We prove that there exists a collection of polynomials $F_{m,j}(\alpha)$ $(j=0,1,\dots,m)$ such that
\begin{equation}
\Cmod m\cong\bigoplus_{\substack{0\le j\le m\\ F_{m,j}(\alpha)\ne0}}\SLmod{2j+1},
\end{equation}
where $\SLmod d$ is the $d$-dimensional irreducible representation of $\Uq$ (Theorem \ref{thm:irred_decomp}),
and show that the polynomials $F_{m,j}(\alpha)$ are written in terms of
a certain $q$-deformation of the hypergeometric polynomials (Theorem \ref{thm:quantum_main}).
Taking a limit $q\to1$, we also obtain the formula \eqref{eq:classical_transition_polynomials} again (Corollary \ref{cor:classical}).

\section{Preliminaries}

We first fix the convention on quantum groups.
We basically follow to \cite{Jimbo1986}, \cite{NYM1993} and \cite{RTF1990}, but modify slightly.

Let $q$ be an indeterminate.
We always discuss over the rational function field $\C(q)$.
The quantum enveloping algebra $\Uq$ is an associative algebra generated by
$k$, $k^{-1}$, $e$, $f$ with the fundamental relations
\begin{align*}
kk^{-1}=k^{-1}k=1,\qquad
kek^{-1}=q^2e,\qquad
kfk^{-1}=q^{-2}f,\qquad
ef-fe=\frac{k-k^{-1}}{q-q^{-1}}.
\end{align*}
$\Uq$ has a (coassociative) coproduct
\begin{align*}
\Delta(k^{\pm1})&=k^{\pm1}\otimes k^{\pm1},\\
\Delta(e)&=e\otimes1+k\otimes e,\\
\Delta(f)&=f\otimes k^{-1}+1\otimes f,
\end{align*}
which enables us to define tensor products of $\Uq$-modules.

The quantum matrix algebra $\Aq$ is an associative algebra
generated by $x_{11},x_{12},x_{21},x_{22}$ with the fundamental relations
\begin{equation}\label{eq:fundamental_relations}
\begin{split}
x_{11}x_{12}&=qx_{12}x_{11},\qquad
x_{21}x_{22}=qx_{22}x_{21},\\
x_{11}x_{21}&=qx_{21}x_{11},\qquad
x_{12}x_{22}=qx_{22}x_{12},\\
x_{12}x_{21}&=x_{21}x_{12},\qquad
x_{11}x_{22}-x_{22}x_{11}=(q-q^{-1})x_{12}x_{21}.
\end{split}
\end{equation}
For convenience, we put
\begin{equation}
z_1=x_{11}x_{22},\quad z_2=x_{12}x_{21}.
\end{equation}
The point is that they \emph{commute}:
\begin{align*}
z_1z_2=z_2z_1.
\end{align*}
The quantum $\alpha$-determinant of size two is then
\begin{equation}
\qadet=x_{11}x_{22}+\alpha qx_{12}x_{21}=z_1+\alpha qz_2.
\end{equation}

\begin{remark}
The quantum $\alpha$-determinant of size two interpolates
the quantum counterparts of the determinant and permanent:
\begin{align*}
\qdet=x_{11}x_{22}-qx_{12}x_{21}=\qadet[-1],\quad
\qper=x_{11}x_{22}+q^{-1}x_{12}x_{21}=\qadet[q^{-2}].
\end{align*}
However, the quantum $\alpha$-determinant of size $n$ does not coincide with
the quantum permanent for any $\alpha$ if $n\ge3$.
This is because $\nu(\cdot)$ is a class function on $\sym n$ in general,
whereas the inversion number $\inv\cdot$ is \emph{not} if $n\ge3$.
\end{remark}

The algebra $\Aq$ becomes a $\Uq$-module by
\begin{equation}\label{eq:module_structure}
\begin{split}
&k^{\pm1}\cdot x_{i1}=q^{\pm1}x_{i1},\qquad
e\cdot x_{i1}=0,\qquad
f\cdot x_{i1}=x_{i2},\\
&k^{\pm1}\cdot x_{i2}=q^{\mp1}x_{i2},\qquad
e\cdot x_{i2}=x_{i1},\qquad
f\cdot x_{i2}=0
\end{split}\qquad\quad(i=1,2).
\end{equation}
These are compatible with the fundamental relations \eqref{eq:fundamental_relations} above.
Our main object is the cyclic submodule of $\Aq$ given by
\begin{equation}
\Cmod m=\Uq\cdot\kakko{\qadet}^m.
\end{equation}

We denote by $\SLmod d$ the $d$-dimensional irreducible representation of $\Uq$.
Notice that
\begin{equation}
\Uq\cdot(x_{11}x_{21})^s\qdet^{m-s}\cong\Uq\cdot(x_{11}x_{21})^s \cong \SLmod{2s+1}.
\end{equation}

Define $q$-analogues of numbers, factorials and binomial coefficients by
\begin{align*}
\q n:=\frac{q^n-q^{-n}}{q-q^{-1}},\qquad
\q n!:=\prod_{i=1}^n\q i,\qquad
\qbinom nk:=\frac{\q n!}{\q k!\q{n-k}!}.
\end{align*}
We recall two well-known identities involved with $q$-binomial coefficients which we will use later.
\begin{itemize}
\item $q$-binomial theorem:
\begin{equation}\label{eq:q-binom-thm}
\prod_{i=1}^n(x+yq^{2i})=\sum_{r=0}^n q^{(n-r)(r+1)}\qbinom nr x^ry^{n-r}.
\end{equation}
\item $q$-Chu-Vandermonde formula:
\begin{equation}\label{eq:q-CV}
\sum_{r=0}^nq^{-r(x+y)}\qbinom x{n-r}\qbinom yr=q^{-ny}\qbinom{x+y}n.
\end{equation}
\end{itemize}

\section{Cyclic modules generated by the quantum alpha-determinant}

\subsection{Some lemmas}

\begin{lem}\label{lem:f-action}
\begin{equation}
f^j\cdot(x_{11}x_{21})^j=q^{-j(j-1)/2}\q{j}!\sum_{r=0}^j q^{-r^2}\qbinom jr^2x_{11}^{j-r}x_{22}^{j-r}(x_{12}x_{21})^r.
\end{equation}
\end{lem}

\begin{proof}
For $1\le i\le 2j$, put
\begin{align*}
f_j(i)=\overbrace{1\otimes\dotsb\otimes1}^{i-1}\otimes f\otimes\overbrace{k^{-1}\otimes\dotsb\otimes k^{-1}}^{2j-i}\in\Uq^{\otimes 2j}.
\end{align*}
Then
\begin{align*}
\Delta^{2j-1}(f)=\sum_{i=1}^{2j}f_j(i),
\end{align*}
so that
\begin{align*}
\Delta^{2j-1}(f)^j=\sum_{1\le n_1,\dots,n_j\le2j}f_j(n_1)\dotsb f_j(n_j).
\end{align*}
Since $f_j(m)f_j(n)=q^{-2}f_j(n)f_j(m)$ if $m>n$ and $f^2\cdot x_{11}=f^2\cdot x_{21}=0$,
we have
\begin{align*}
\Delta^{2j-1}(f)^j&=q^{-j(j-1)/2}\q j!
\sum_{1\le n_1<\dotsb<n_j\le2j}f_j(n_1)\dotsb f_j(n_j)+R,
\end{align*}
where $R$ is a certain element in $\Uq^{\otimes 2j}$ such that $R\cdot(x_{11}^jx_{21}^j)=0$.
Here we also use the well-known identity
\begin{align*}
\sum_{\sigma\in\sym j}x^{\inv\sigma}=(1+x)(1+x+x^{2})\dotsb(1+x+\dotsb+x^{j-1})
\end{align*}
with $x=q^{-2}$. Now we consider
\begin{align*}
f_j(n_1)\dotsb f_j(n_j)\cdot(x_{11}^jx_{21}^j)
\end{align*}
for given $n_1,\dots,n_j$ ($1\le n_1<\dotsb<n_j\le2j$).
Suppose that
\begin{align*}
n_1<\dots<n_r\le j<n_{r+1}<\dots<n_j,
\end{align*}
for some $r$ and define $m_1,\dots,m_r$ ($1\le m_1<\dotsb<m_r\le j$) by the condition
\begin{align*}
\ckakko{n_{r+1},n_{r+2},\dots,n_j}\sqcup\ckakko{j+m_1,j+m_2,\dots,j+m_r}=\ckakko{j+1,j+2,\dots,2j}.
\end{align*}
Then we have
\begin{align*}
f_j(n_1)\dotsb f_j(n_j)\cdot(x_{11}^jx_{21}^j)
&=q^{\beta}\cdot \underbrace{x_{11}\dotsb\overset{n_1}{\mathstrut x_{12}}\dotsb\overset{n_r}{\mathstrut x_{12}}\dotsb x_{11}}_j
\cdot \underbrace{x_{22}\dotsb\overset{m_1}{\mathstrut x_{21}}\dotsb\overset{m_r}{\mathstrut x_{21}}\dotsb x_{22}}_j\\
&=q^{\beta+\gamma}\cdot x_{11}^{j-r}\cdot x_{12}^r\cdot x_{22}^{j-r}\cdot x_{21}^r\\
&=q^{\beta+\gamma+r(j-r)}x_{11}^{j-r}x_{22}^{j-r}(x_{12}x_{21})^r,
\end{align*}
where $\beta$ and $\gamma$ are calculated as
\begin{align*}
\beta&=-\ckakko{(2j-n_1)+(2j-n_2)+\dotsb+(2j-n_j)-(1+2+\dotsb+j-1)}+(1+2+\dotsb+j-1)\\
&=\frac{j(j-1)}2-rj+(n_1+\dotsb+n_r)-(m_1+\dotsb+m_r),\\
\gamma&=\ckakko{(j-m_r)+(j-1-m_{r-1})+\dotsb+(j-r+1-m_1)}\\
&\eqsp\qquad-\ckakko{(j-n_r)+(j-1-n_{r-1})+\dotsb+(j-r+1-n_1)}\\
&=(n_1+\dotsb+n_r)-(m_1+\dotsb+m_r).
\end{align*}
Thus we get
\begin{align*}
f_j(n_1)\dotsb f_j(n_j)\cdot(x_{11}^jx_{21}^j)
=q^{-r^2+j(j-1)/2+2(n_1+\dotsb+n_r)-2(m_1+\dotsb+m_r)}x_{11}^{j-r}x_{22}^{j-r}(x_{12}x_{21})^r.
\end{align*}
Using this, we have
\begin{align*}
f^j\cdot(x_{11}^jx_{22}^j)
&=q^{-j(j-1)/2}\q j!\sum_{1\le n_1<\dotsb<n_j\le2j}f_j(n_1)\dotsb f_j(n_j)\cdot(x_{11}^jx_{21}^j)\\
&=\q j!\sum_{r=0}^jq^{-r^2}\sum_{\substack{1\le n_1<\dotsb<n_r\le j\\ 1\le m_1<\dotsb<m_r\le j}}
q^{2(n_1+\dotsb+n_r)}q^{-2(m_1+\dotsb+m_r)}x_{11}^{j-r}x_{22}^{j-r}(x_{12}x_{21})^r\\
&=\q j!\sum_{r=0}^jq^{-r^2}
e_r(1,q^2,\dots,q^{2(j-1)})e_r(1,q^{-2},\dots,q^{-2(j-1)})
x_{11}^{j-r}x_{22}^{j-r}(x_{12}x_{21})^r,
\end{align*}
where $e_r(x_1,x_2,\dots,x_j)$ denotes the $r$-th elementary symmetric polynomial in $x_1,x_2,\dots,x_j$.
Using the identity (see, e.g. \cite{Mac})
\begin{align*}
e_r(1,q^2,\dots,q^{2j-2})=q^{r(j-1)}\qbinom jr
\end{align*}
together with the symmetry $\qbinom[q]jr=\qbinom[q^{-1}]jr$, we obtain
\begin{align*}
f^j\cdot(x_{11}^jx_{21}^j)
&=\q j!\sum_{r=0}^jq^{-r^2}\qbinom jr^{\!2}
x_{11}^{j-r}x_{22}^{j-r}(x_{12}x_{21})^r.
\end{align*}
Since $(x_{11}x_{21})^j=q^{-j(j-1)/2}x_{11}^jx_{21}^j$, we have the desired conclusion.
\end{proof}

It is straightforward to verify the relations
\begin{equation}
\begin{split}
z_1\cdot x_{22}&=x_{22}\cdot(z_1+(q^3-q)z_2),\\
z_2\cdot x_{22}&=q^2x_{22}\cdot z_2.
\end{split}\label{eq:easy_basic}
\end{equation}
Using this, we get the
\begin{lem}\label{lem:x2z}
\begin{equation}
x_{11}^lx_{22}^l=\prod_{r=1}^l\kakko{z_1+(q^{2r-1}-q)z_2}.
\end{equation}
\end{lem}

\begin{proof}
By \eqref{eq:easy_basic}, it follows that
\begin{align*}
(z_1+(q^{2r-1}-q)z_2)\cdot x_{22}=x_{22}\cdot(z_1+(q^{2r+1}-q)z_2),
\end{align*}
by which the lemma is proved by induction on $l$.
\end{proof}

\subsection{Irreducible decomposition}

\begin{thm}\label{thm:irred_decomp}
There exists a collection of $\C(q)$-valued functions $F_{m,j}(\alpha)$ $(j=0,1,\dots,m)$ such that
\begin{equation}
\Cmod m\cong\bigoplus_{\substack{0\le j\le m\\ F_{m,j}(\alpha)\ne0}}\SLmod{2j+1}.
\end{equation}
\end{thm}

\begin{proof}
Notice that
\begin{align}
\kakko{\qadet}^m=\sum_{j=0}^m \binom mj(\alpha q)^jz_1^{m-j}z_2^j
\end{align}
is a homogeneous polynomial of degree $m$ in the commuting variables $z_1$ and $z_2$.
On the other hand, it is also clear that the vectors
\begin{align*}
v_{m,j}=\kakko{f^j\cdot(x_{11}x_{21})^j}\qdet^{m-j}\qquad(j=0,1,\dots,m)
\end{align*}
are linearly independent (since $e^j\cdot v_{m,j}\ne0$ and $e^{j+1}\cdot v_{m,j}=0$),
and they are homogeneous polynomials of degree $m$ in $z_1$ and $z_2$ by Lemma \ref{lem:f-action}.
Therefore, $\{v_{m,j}\}_{j=0}^m$ form a basis of the space consisting of
the homogeneous polynomials of degree $m$ in $z_1$ and $z_2$,
so that there exist $\C(q)$-valued functions $F_{m,j}(\alpha)$ such that
\begin{equation}\label{eq:expansion}
\kakko{\qadet}^m=\sum_{j=0}^m F_{m,j}(\alpha)v_{m,j}.
\end{equation}
Since $\Uq\cdot v_{m,j}\cong\SLmod{2j+1}$, this proves the theorem.
\end{proof}

The conditions for the functions $F_{m,j}(\alpha)$ are described in terms of polynomials in $\C(q)[z_1,z_2]$
by virtue of Lemmas \ref{lem:f-action} and \ref{lem:x2z}.
Since $z_1$ and $z_2$ commute, it is meaningful to consider the specialization $z_1=z$, $z_2=1$,
where $z$ is a new indeterminate.
Put
\begin{align}
g_j(z)&=\prod_{i=1}^j(z+q^{2i-1}-q),\\
v_j(z)&=q^{-j(j-1)/2}\q{j}!\sum_{r=0}^j q^{-r^2}\qbinom jr^2 g_{j-r}(z).
\end{align}
Then \eqref{eq:expansion} together with Lemmas \ref{lem:f-action} and \ref{lem:x2z} yields
\begin{equation}\label{eq:base_formula}
(z+q\alpha)^m=\sum_{j=0}^m F_{m,j}(\alpha)v_j(z)(z-q)^{m-j}.
\end{equation}
If we take the $l$-th derivative of this formula with respect to $z$ ($l=0,1,\dots,m$), then we have
\begin{align*}
\frac{m!}{(m-l)!}(z+q\alpha)^{m-l}=\sum_{j=0}^m F_{m,j}(\alpha)\sum_{t=0}^l\binom lt v_j^{(l-t)}(z)\frac{(m-j)!}{(m-j-t)!}(z-q)^{m-j-t}.
\end{align*}
Letting $z=q$, we get the relation
\begin{equation}\label{eq:relation_among_F}
\binom mlq^{m-l}(1+\alpha)^{m-l}
=\sum_{j=m-l}^m F_{m,j}(\alpha)\frac{v_j^{(l-m+j)}(q)}{(l-m+j)!}
=\sum_{s=0}^l F_{m,m-s}(\alpha)\frac{v_{m-s}^{(l-s)}(q)}{(l-s)!}.
\end{equation}

Now we calculate $v_j^{(i)}(q)/i!$.
By the $q$-binomial theorem \eqref{eq:q-binom-thm}, we have
\begin{align*}
g_j(z)=\sum_{i=0}^j q^{j(j-i)}\qbinom ji(z-q)^i.
\end{align*}
Hence it follows that
\begin{align*}
v_j(z)&=q^{-j(j-1)/2}\q{j}!\sum_{r=0}^j q^{-r^2}\qbinom jr^2 g_{j-r}(z)\\
&=q^{j(j+1)/2}\q{j}!\sum_{i=0}^jq^{-ij}\ckakko{\sum_{r=0}^{j-i}q^{r(i-2j)}\qbinom jr^2\qbinom{j-r}i}(z-q)^i\\
&=q^{j(j+1)/2}\q{j}!\sum_{i=0}^jq^{-ij}\ckakko{\qbinom ji\sum_{r=0}^{j-i}q^{r(i-2j)}\qbinom j{j-r}\qbinom{j-i}r}(z-q)^i.
\end{align*}
Using the $q$-Chu-Vandermonde formula \eqref{eq:q-CV}, we get
\begin{align*}
\sum_{r=0}^{j-i}q^{r(i-2j)}\qbinom j{j-r}\qbinom{j-i}r=q^{j(i-j)}\qbinom{2j-i}j.
\end{align*}
This provides
\begin{equation}
v_j(z)=q^{-j(j-1)/2}\sum_{i=0}^j\frac{\q j!\q{2j-i}!}{\q i!\q{j-i}!^2}(z-q)^i,
\end{equation}
or
\begin{equation}
\frac{v_j^{(i)}(q)}{i!}=q^{-\binom j2}\frac{\q j!\q{2j-i}!}{\q i!\q{j-i}!^2}.
\end{equation}
Thus the formula \eqref{eq:relation_among_F} is rewritten more explicitly as
\begin{equation}\label{eq:relation_among_F2}
\q{m-l}!^2\binom mlq^{m-l}(1+\alpha)^{m-l}
=\sum_{s=0}^l q^{-\binom{m-s}2}\frac{\q{m-s}!\q{2m-l-s}!}{\q{l-s}!}F_{m,m-s}(\alpha).
\end{equation}

\subsection{Expression of $F_{m,j}(\alpha)$ in terms of mixed hypergeometric polynomials}

From \eqref{eq:relation_among_F} (or \eqref{eq:relation_among_F2}), we can conclude that
$F_{m,j}(\alpha)$ is a \emph{polynomial function} in $\alpha$
which is divisible by $(1+\alpha)^{j}$, that is
\begin{align}\label{eq:common_factor}
F_{m,j}(\alpha)=(1+\alpha)^jQ_{m,j}(\alpha)
\end{align}
for some $Q_{m,j}(\alpha)\in\C(q)[\alpha]$.
By \eqref{eq:common_factor} and \eqref{eq:relation_among_F2}, we have
\begin{equation}\label{eq:key-relation}
\qbinom{2m-2i}{m-i}^{-1}\binom miq^{m-i}
=\sum_{j=0}^i\qbinom{2i-2m-1}{i-j}
(-1)^{i-j}(1+\alpha)^{i-j}\cdot q^{-\binom{m-j}2}\q{m-j}!Q_{m,m-j}(\alpha).
\end{equation}
To solve this,
we need the following lemma.
\begin{lem}
\begin{equation}\label{eq:inverse}
\kakko{\qbinom{2i-2m-1}{i-j}}_{0\le i,j\le m}^{\!-1}
=\kakko{\frac{\q{2m-2i+1}}{\q{2m-2j+1}}\qbinom{2m-2j+1}{i-j}}_{0\le i,j\le m}.
\end{equation}
\end{lem}

\begin{proof}
We should prove
\begin{equation}\label{eq:to-be-proved}
\sum_{k=j}^i\qbinom{2i-2m-1}{i-k}\cdot\frac{\q{2m-2k+1}}{\q{2m-2j+1}}\qbinom{2m-2j}{k-j}=\delta_{ij}
\end{equation}
for $0\le j\le i\le m$ since the matrices in \eqref{eq:inverse} are lower triangular.
The case where $i=j$ is clear. Assume that $i>j$.
By putting $d=i-j$, $n=m-j$ and changing the running index by $r=k-j$,
\eqref{eq:to-be-proved} is reduced to
\begin{equation}
\sum_{r=0}^d\q{2n+1-2r}\qbinom{2d-(2n+1)}{d-r}\qbinom{2n+1}r=0\quad(0<d\le n).
\end{equation}
To prove this, it suffices to show that the function
\begin{align*}
f(x)=\sum_{r=0}^d\q{x-2r}\qbinom{2d-x}{d-r}\qbinom{x}r
\end{align*}
is constant, which is readily seen to be zero.
Notice that $f(x)$ is a rational function in $z=q^x$, and its numerator is a polynomial in $z$ of degree at most $2d$.
Hence it is enough to verify that $f(l)=0$ for $l=0,1,\dots,2d$.
However, since we easily see that $f(x)+f(2d-x)=0$, which also implies $f(d)=0$,
we have only to check $f(l)=0$ for $l=0,1,\dots,d-1$.

Let $l\in\Z$ such that $0\le l<d$. Then we have
\begin{align*}
f(l)&=\sum_{r=0}^l\q{l-2r}\qbinom{2d-l}{d-r}\qbinom lr\quad(\because \qbinom lr=0 \text{ if } l<r\le d)\\
&=\sum_{s=0}^l\q{l-2(l-s)}\qbinom{2d-l}{d-(l-s)}\qbinom l{l-s}\\
&=\sum_{s=0}^l\q{-(l-2s)}\qbinom{2d-l}{d-s}\qbinom ls=-f(l),
\end{align*}
which means $f(l)=0$.
This completes the proof of the lemma.
\end{proof}

Now we define the \emph{mixed hypergeometric series} by
\begin{equation}\label{eq:mixedHG}
\mixhgf{a_1,\dots,a_k}{b_1,\dots,b_l}{c_1,\dots,c_m}{d_1,\dots,d_n}{x}
=\sum_{i=0}^\infty\frac{\ps{a_1}i\dotsb\ps{a_k}i}{\ps{b_1}i\dots\ps{b_l}i}\frac{\qps{c_1}i\dotsb\qps{c_m}i}{\qps{d_1}i\dotsb\qps{d_m}i}\frac{x^i}{\q i!},
\end{equation}
where $\ps ai=a(a+1)\dotsb(a+i-1)$ and $\qps ai=\q{a}\q{a+1}\dotsb\q{a+i-1}$ (cf. \cite{KK1989AMV}).

\begin{thm}\label{thm:quantum_main}
For $s=0,1,\dots,m$,
\begin{equation}
F_{m,s}(\alpha)=q^{\binom{s+1}2}\binom ms\frac{\q{s}!}{\q{2s}!}
(1+\alpha)^s\,\mixhgf{s-m}{s+1}{s+1,s+1}{2s+2}{q(1+\alpha)}
\end{equation}
holds.
\end{thm}

\begin{proof}
By \eqref{eq:key-relation} and \eqref{eq:inverse}, we have
\begin{align*}
Q_{m,m-i}(\alpha)
&=\frac{q^{\binom{m-i}2}}{\q{m-i}!}
\sum_{j=0}^i(-1)^{i-j}q^{m-j}\frac{\q{2m-2i+1}}{\q{2m-2j+1}}\qbinom{2m-2j+1}{i-j}
\qbinom{2m-2j}{m-j}^{-1}\binom mj(1+\alpha)^{i-j}\\
&=\frac{q^{\binom{m-i+1}2}m!\q{2m-2i+1}}{\q{m-i}!}
\sum_{r=0}^i(-q)^r\frac{\q{m-i+r}!^2}{(m-i+r)!(i-r)!\q{2m-2i+r+1}!}\frac{(1+\alpha)^r}{\q{r}!}.
\end{align*}
Since
\begin{align*}
(i-r)!=(-1)^r\frac{i!}{\ps{-i}r},\qquad
(n+r)!=n!\ps{n+1}r,\qquad
\q{n+r}!=\q{n}!\qps{n+1}r,
\end{align*}
we have
\begin{equation}
\begin{split}
Q_{m,m-i}(\alpha)
&=\frac{q^{\binom{m-i+1}2}m!\q{m-i}!}{i!(m-i)!\q{2m-2i}!}
\sum_{r=0}^i\frac{\ps{-i}r\qps{m-i+1}r^2}{\ps{m-i+1}r\qps{2m-2i+2}r}\frac{(q(1+\alpha))^r}{\q{r}!}\\
&=q^{\binom{m-i+1}2}\binom mi\frac{\q{m-i}!}{\q{2m-2i}!}
\mixhgf{-i}{m-i+1}{m-i+1,m-i+1}{2m-2i+2}{q(1+\alpha)}.
\end{split}\label{eq:formula_for_Q}
\end{equation}
If we substitute this into \eqref{eq:common_factor} and replace $m-i$ by $s$,
then we have the conclusion.
\end{proof}

\begin{remark}
The function $\Phi$ given by \eqref{eq:mixedHG} satisfies the difference-differential equation
\begin{align*}
&\Bigl\{-(E+a_1)\dotsb(E+a_k)\q{E+c_1}\dotsb\q{E+c_m}\\
&\qquad{}+\partial_q(E+b_1-1)\dotsb(E+b_l-1)\q{E+d_1-1}\dotsb\q{E+d_n-1}\Bigr\}\Phi=0,
\end{align*}
where we put
\begin{align*}
E=x\frac d{dx},\quad
\q{E+a}=\frac{q^{E+a}-q^{-E-a}}{q-q^{-1}},\quad
\partial_q f(x)=\frac{f(qx)-f(q^{-1}x)}{qx-q^{-1}x}.
\end{align*}
If we take a limit $q\to1$, then the difference-differential equation above becomes
a hypergeometric differential equation for $\genhgf{k+m}{l+n}{a_1,\dots,c_m}{b_1,\dots,d_n}x$.
\end{remark}

\subsection{Classical case}

All the discussion above also work in the classical case (i.e. the case where $q=1$).
Thus, by taking a limit $q\to1$ in Theorem \ref{thm:quantum_main},
we will obtain Theorem 4.1 in \cite{KMW:TAMS} again.
We abuse the same notations used in the discussion of quantum case above to indicate the classical counterparts.
From \eqref{eq:formula_for_Q}, we have
\begin{equation}
\begin{split}
Q_{m,s}(\alpha)
&=\frac{m!}{(m-s)!(2s)!}\,\genhgf32{s-m,s+1,s+1}{s+1,2s+2}{1+\alpha}\\
&=\frac{m!}{(m-s)!(2s)!}\,\genhgf21{s-m,s+1}{2s+1}{1+\alpha}.
\end{split}
\end{equation}
Notice that
\begin{equation}
\genhgf21{s-m,s+1}{2s+2}{1-x}=\frac{m!(2s+1)!}{s!(m+s+1)!}\genhgf21{s-m,s+1}{-m}{x}.
\end{equation}
Thus we also get
\begin{equation}
Q_{m,s}(\alpha)=
\frac{m!^2(2s+1)}{(m-s)!s!(m+s+1)!}\,
\genhgf21{s-m,s+1}{-m}{-\alpha}\qquad(s=0,1,\dots,m).
\end{equation}
Summarizing these, we have the
\begin{cor}[Classical case]\label{cor:classical}
\begin{equation}
\begin{split}
F_{m,s}(\alpha)
&=\frac{m!}{(m-s)!(2s)!}(1+\alpha)^s\,\genhgf21{s-m,s+1}{2s+1}{1+\alpha}\\
&=\frac{\binom{2m}{m-s}-\binom{2m}{m-s-1}}{\binom{2m}m s!}
(1+\alpha)^s\,\genhgf21{s-m,s+1}{-m}{-\alpha}
\end{split}\label{eq:classical_again}
\end{equation}
for $s=0,1,\dots,m$.
\qed
\end{cor}

\begin{ackn}
The author would like to thank Max-Planck-Institut f\"ur Mathematik for the support and hospitality.
\end{ackn}


\bigskip
\noindent
Department of Mathematical Sciences, University of the Ryukyus\\
1 Senbaru, Nishihara-cho, Okinawa 903-0213 Japan

\smallskip
\noindent
\texttt{kimoto@math.u-ryukyu.ac.jp}

\smallskip
\noindent
Max-Planck-Institut f\"ur Mathematik\\
Vivatsgasse 7, 53111 Bonn, Germany

\smallskip
\noindent
\texttt{kimoto@mpim-bonn.mpg.de}

\end{document}